\documentclass[reqno,12pt]{amsart}

\usepackage{epsf}
\usepackage{graphics}
\usepackage{graphicx}
\usepackage{amssymb}
\usepackage{amsmath}

\date{}

\theoremstyle{plain}
\newtheorem{theorem}{Theorem}

\newtheorem{proposition}{Proposition}
\newtheorem{lemma}{Lemma}
\newtheorem{rem}{Remark}

\theoremstyle{definition}

\theoremstyle{remark}

\newtheorem*{remark}{Remark}

\def\N{{\mathbb N}}
\def\Z{{\mathbb Z}}
\def\Q{{\mathbb Q}}
\def\R{{\mathbb R}}
\def\H{{\mathbb H}}

\title{On the crossing number of arithmetic curve systems}

\author{Sebastian Baader, Claire Burrin, Luca Studer}

\begin{document}

\begin{abstract}
We show that the family of systoles of hyperbolic surfaces associated with congruence lattices in $\text{SL}_2(\Z)$ have asymptotically minimal crossing number.
\end{abstract}


\maketitle

\section{Introduction}

A famous question in surface theory asks for the maximal number of systoles among all closed hyperbolic surfaces of fixed genus~$g$. The known examples with the largest number of systoles are due to Schmutz-Schaller~\cite{Sch} and are associated with congruence lattices $\Gamma(N) \triangleleft \text{SL}_2(\Z)$. The lattices $\Gamma(N)$, indexed by a natural number $N$, give rise to a family of surfaces $X(N)=\Gamma(N) \backslash \H$ of genus about $N^3$. The number of systoles of $X(N)$ is significantly higher than the genus, namely about $N^4=g^{4/3}$. It is not known whether these surfaces are maximal with respect to the number of systoles. Indeed, the best upper bound on the number of systoles on a closed hyperbolic surface of genus~$g$ is $g^2/\log(g)$, as derived by Parlier in~\cite{Pa}. The purpose of this note is to show that the systems of systoles $\mathcal{C}(N)$ of the surfaces $X(N)$ are asymptotically optimal in terms of the crossing number.

Let $\mathcal{C}=\{c_1,\ldots,c_n\}$ be a family of pairwise non-homotopic simple closed curves on a surface $\Sigma$. We define the crossing number of the family $\mathcal{C}$ as the total number of pairwise intersections,
$$\text{cr}(\mathcal{C})=\sum_{k<l} i(c_k,c_l),$$
where $i(c_k,c_l)$ denotes the usual geometric intersection number of the curves $c_k$ and $c_l$, as defined in Chapter~1 of~\cite{FM}.  

\begin{theorem}
\label{crossingnumber}
The family of systoles $\mathcal{C}(N)$ of the surfaces $X(N)$ satisfies
$$\liminf_{N \to \infty} \frac{\log(\text{cr}(\mathcal{C}(N)))}{\log(N)}=5.$$
Moreover, every other family $\mathcal{D}(N)$ of pairwise non-homotopic simple closed curve systems in the same surface $X(N)$ with $|\mathcal{D}(N)|=|\mathcal{C}(N)|$ satisfies
$$\liminf_{N \to \infty} \frac{\log(\text{cr}(\mathcal{D}(N)))}{\log(N)} \geq 5.$$
\end{theorem}

As mentioned above, the genus $g(X(N))$ and the number of systoles $|\mathcal{C}(N)|$ are about $N^3$ and $N^4$, respectively. Moreover, the surfaces $X(N)$ have about $N^2$ cusps. If we cap off these cusps with tori, we obtain systems of pairwise non-homotopic simple closed curves on closed topological surfaces whose genus is still about $N^3$. Therefore, the second statement of Theorem~\ref{crossingnumber} is equivalent to: all curve systems with about $g^{4/3}$ curves on a closed surface of genus $g$ have crossing number at least $g^{5/3}$. In the next section, we prove a generalisation of this statement, where $g^{4/3}$ and $g^{5/3}$ are replaced by $g^{1+\alpha}$ and $g^{1+2\alpha}$, respectively. This is a purely topological fact. In contrast, the determination of the crossing number $\text{cr}(\mathcal{C}(N))$, presented in Section~3, is of arithmetic nature. We conclude the note with a discussion on the crossing number of curve systems with $g^{1+\alpha}$ curves, for general $\alpha \leq 1/3$.

\section*{Acknowledgments}
We would like to thank Hugo Parlier and Jasmin J\"org for sharing ideas on curve systems --- in particular Hugo for coining the expression `crossing number' for curve systems --- and James Rickards for discussions on explicit formulas for intersection numbers.

\section{Lower bound on the crossing number}

In this short section, we derive an upper bound of order $g^{1+\alpha}$ on the number of simple closed curves with a total number of $g^{1+2\alpha}$ double points, on the standard closed surface $\Sigma_g$ of genus~$g$.

\begin{proposition}
\label{doublepoints}
Let $\alpha \geq 0$, and let $\mathcal{C}_g$ be a sequence of systems of pairwise non-homotopic simple closed curves on the closed surface of genus~$g$, with crossing number bounded above by $g^{1+2\alpha}$. Then the number of curves $N(g)=|\mathcal{C}_g|$ satisfies
$$\limsup_{g \to \infty} \frac{N(g)}{g^{1+\alpha}} \leq 6.$$
\end{proposition}

The special case $\alpha=1/3$ implies that every curve system with about $g^{4/3}$ curves on $\Sigma_g$ has crossing number at least $g^{5/3}$. This is the second statement of Theorem~\ref{crossingnumber}, formulated in terms of $g=N^3$.

The proof of Proposition~\ref{doublepoints} relies on the fact that a closed surface of genus~$g$ contains at most $3g-3$ pairwise non-homotopic disjoint simple closed curves. In fact, the special case $\alpha=0$ is a direct consequence of this: every curve system on $\Sigma_g$ with a total of at most $g$ double points has at most $4g-3$ curves. In this case, the upper bound~$6$ in the proposition can be replaced by~$4$.

\begin{proof}[Proof of Proposition~\ref{doublepoints}]
We fix $\alpha>0$, as the case $\alpha=0$ is already settled. 
Let $\mathcal{C}_g$ be a sequence of systems of pairwise non-homotopic simple closed curves on a closed surface of genus~$g$, with crossing number bounded above by $g^{1+2\alpha}$. Assume that the number of curves $N(g)=|\mathcal{C}_g|$ satisfies
$$N(g) \geq 6 g^{1+\alpha},$$
for infinitely many $g \in \N$.
We will show that this leads to a contradiction.
First, the number of curves in $\mathcal{C}_g$ with at least
$g^{\alpha}$ double points is at most
$$\frac{2g^{1+2\alpha}}{g^{\alpha}}=2g^{1+\alpha},$$
since every double point involves precisely two curves.
This leaves us with a subset $\widetilde{\mathcal{C}}_g \subset \mathcal{C}_g$ of at least $(6-2)g^{1+\alpha}=4g^{1+\alpha}$ curves, all of which have at most $g^{\alpha}$ double points. Now comes the elegant conclusion.

Choose a maximal set of pairwise disjoint curves among the system $\widetilde{\mathcal{C}}_g$. Call these curves $c_1,c_2,\ldots,c_n$; observe $n \leq 3g-3$. Now every other curve of the system $\widetilde{\mathcal{C}}_g$ intersects at least one of the curves $c_i$. Keeping in mind that every curve $c_i$ intersects at most $g^{\alpha}$ other curves of the system $\widetilde{\mathcal{C}}_g$, we deduce that $\widetilde{\mathcal{C}}_g$ has a total of at most
$$3g-3+(3g-3) \cdot g^{\alpha} \leq 3g+3g^{1+\alpha}$$
curves. On the other hand, by construction, $\widetilde{\mathcal{C}}_g$ has at least
$4g^{1+\alpha}$ curves. With this contradiction, for large $g$, we conclude the proof of Proposition~\ref{doublepoints}.
\end{proof}

\section{Counting intersections of systoles}

Let $N>2$. There is a one-to-one correspondence between closed geodesics on the surface $X(N)$ and conjugacy classes of hyperbolic elements $A\in \Gamma(N)$. The length $\ell$ of the closed geodesic associated to the conjugacy class of $A\in \Gamma(N)$ is given by the identity
$$
2\cosh(\ell/2) = |{\rm tr}(A)|.
$$ 
Each element of $\Gamma(N)$ is a matrix $A\in \mathrm{SL}_2(\Z)$ such that $A\equiv I$ (mod $N$). As such it can be written as $A=I-NB$ for some matrix $B\in M_2(\Z)$. A direct computation shows that the condition $\det(A)=1$ is equivalent to ${\rm tr}(B)=N\det(B)$. Since 
$$
{\rm tr}(A) = 2-N\cdot {\rm tr}(B) = 2-N^2 \det(B)
$$ 
we see that if $A$ is hyperbolic, $|{\rm tr}(A)|$ is minimized when $B\in\mathrm{SL}_2(\Z)$. We conclude that systoles on $X(N)$ correspond to conjugacy classes of elements $A\in \Gamma(N)$ for which $|{\rm tr}(A)|=N^2-2$. 

This observation is the starting point of the Schmutz-Schaller correspondence \cite{Sch} that we now review. From the discussion above, we have a bijection 
\begin{align*}
     \{B\in \Gamma(1): {\rm tr}
(B)=N \} \overset{\Phi}{\longrightarrow} \{ A \in \Gamma(N): {\rm tr}
(A)=2-N^2\}
\end{align*}
given by $\Phi(B)=I-NB$. By Cayley--Hamilton, we may write this map as $\Phi(B)=-B^2$. The map $\Phi$ commutes with the action of $\Gamma(1)$ by conjugation, inducing a one-to-one correspondence between modular geodesics (closed geodesics on $X(1)$) of trace $N$ and isometry classes (under $\Gamma(1)$) of systoles on $X(N)$. There is another classical one-to-one correspondence, going back to Gauss, between modular geodesics and binary integral quadratic forms of positive discriminant. To be more precise, under the latter, the number of primitive modular geodesics of trace $N$ is equal to the class number $h(N^2-4)$, that is the number of (narrow) classes of primitive binary integral quadratic forms of discriminant $D=N^2-4$, or equivalently, is equal to twice the number of ideal classes in the ring of integers of the real quadratic field $\Q(\sqrt{D})$. Hence, up to a multiplicative constant, we have
\begin{align}\label{SS}
|\mathcal{C}(N)| = [\Gamma(1):\Gamma(N)]\cdot h(N^2-4).
\end{align}
The index of $\Gamma(N)$ in $\Gamma(1)$ is roughly of order $N^3$. More precisely,
\begin{lemma}
$$\liminf_{N\to\infty}\, \frac{\log[\Gamma(1):\Gamma(N)]}{\log N} = 3.$$
\end{lemma}
\begin{proof} 
    We have the exact formula $[\Gamma(1):\Gamma(N)] = N^3 \prod_{p\mid N} (1-p^{-2})$. Trivially, the index is bounded above by $N^3$. On the other hand taking the logarithm of this expression gives
    $$
    3\log N + \sum_{p\mid N} \log(1-p^{-2}) =  3\log N - \sum_{p\mid N} \sum_{r\geq 1} \frac{1}{rp^{2r}}.
    $$
    The double sum is bounded above by
    $
     \sum_{r=1}^\infty \frac{2^{-2r}}{r} \omega(N) = \log(4/3) \omega(N),
    $
    where $\omega(N)$ is the number of distinct prime divisors of $N$. We can now apply the standard bound $\omega(N)\ll \tfrac{\log N}{\log\log N}$.
    \end{proof}

Schmutz-Schaller uses the identity (\ref{SS}) to show that $|\mathcal{C}(N)|$, the number of systoles on $X(N)$, is roughly of order $N^4$. We note the following refinement of his result.

\begin{proposition}
Assuming that the Generalized Riemann Hypothesis (GRH) is true, we have
$$
\frac{N}{\log N}(\log\log N)^{-1} \ll \frac{|\mathcal{C}(N)|}{g(X(N))} \ll \frac{N}{\log N}\log\log N
$$
for all $N$ such that $N^2-4$ is squarefree.
\end{proposition}
\begin{proof}
    We first recall the exact formula
$$
g(X(N)) = 1 + [\Gamma(1):\Gamma(N)] \left(\frac{1}{24}- \frac{1}{4N}\right)
$$
when $N>2$. Then
    $$
    \frac{|\mathcal{C}(N)|}{g(X(N))} = \frac{|\mathcal{C}(N)|}{[\Gamma(1):\Gamma(N)]} \left(\frac{1}{24}+O(N^{-1})\right) = h(N^2-4)(1+O(N^{-1}))
    $$
    up to a multiplicative constant. We may assume that $N$ is chosen such that $D=N^2-4$ is squarefree.  Dirichlet's class number formula states that 
\begin{align}\label{Dirichlet}
h(D)\log\epsilon _D = \sqrt{D} L(\chi_D,1),
\end{align}
where $\chi_D$ is the Kronecker symbol modulo $D$, $L(\chi_D,1)$ is the special $L$-value $L(\chi_D,1)=\sum_{n=1}^\infty \tfrac{\chi_D(n)}{n}$ and $\epsilon_D$ is the fundamental unit of the ring of integers of the real quadratic field $\Q(\sqrt{D})$. Most progress on estimating $h(D)$ has been achieved by estimating the size of $L(\chi_D,1)$ by analytic means. The best known bounds are
\begin{align}\label{lower/upper bounds}
D^{-\theta}\ll_\theta |L(\chi_D,1)| \ll \log D
\end{align}
and upon assuming GRH, these bounds can be sharpened to 
$$
(\log\log D)^{-1}\ll |L(\chi_D,1)| \ll \log \log D.
$$

We do not know in general how to separate the two terms on the left on (\ref{Dirichlet}). However for fields of discriminant of Richaud--Degert type, such as $D=N^2-4$, the fundamental unit is known to be small, and in fact can be expressed as $\epsilon_D= \frac{N+\sqrt{D}}{2}\asymp \sqrt{D}$. The result follows by collecting these estimates.
\end{proof}

We extend the Schmutz-Schaller correspondence to deduce the following upper bound on the crossing number from the recent work of Jung and Sardari \cite{JungSardari2022} on  intersections of modular geodesics.

\begin{proposition}
For $N$ sufficiently large and such that $N^2-4$ is squarefree, we have ${\rm cr}(\mathcal{C}(N))\ll [\Gamma(1):\Gamma(N)] (N\log N)^2$.
\end{proposition}

\begin{proof}
For each hyperbolic $M\in \mathrm{SL}_2(\R)$ let $C_M$ denote the hyperbolic geodesic in $\mathbb{H}$ connecting the fixed points of $M$. A direct computation suffices to check that $M$ and $I-NM$ (for any $N$) have the same fixed points; writing $M=(\begin{smallmatrix} a& b \\ c&d \end{smallmatrix})$, we have
 \begin{align*}
     (I-NM) x= x \iff \frac{x-N(ax+b)}{1-N(cx+d)}=x \iff Mx = x.
 \end{align*}
 Hence we can extend $\Phi\otimes \Phi$ to a bijection between pairs $(B,B')$ for which $C_B$ and $C_{B'}$ intersect and pairs $(A,A')$ for which $C_A$ and $C_{A'}$ intersect. We consider further the bijection induced by $\Phi\otimes\Phi$ between pairs of conjugacy classes (under the action $\Gamma(1)\times \Gamma(1)$) for which there exists a pair of representatives with intersecting geodesics. To avoid overcounting, we will further identify on both sides pairs that coincide up to simultaneous $\Gamma(1)$-conjugation. It follows that the number of intersecting systoles on $X(N)$ is bounded above by the number of intersections of modular geodesics of trace $N$ times $|\Delta \setminus (\Gamma(1)\times \Gamma(1))/(\Gamma(N)\times \Gamma(N))|= [\Gamma(1):\Gamma(N)]$. The recent work of Jung and Sardari shows that the number of intersections for modular geodesics of discriminant $D=N^2-4$ is $\asymp (h(D)\log \epsilon_D)^2$ for $N$ sufficiently large; see \cite[Theorem 1.8]{JungSardari2022}. By the class number formula, this is $(N^2-4)L(\chi,1)^2$ and we could now rely on the (elementary) upper bound stated in (\ref{lower/upper bounds}). For the convenience of the reader, we include a self-contained proof. By definition, we have 
 $$
 L(\chi_D,1)=\sum_{n=1}^\infty  \left(\frac{D}{n}\right)\frac{1}{n}.
 $$
 For partial sums ($1<\Delta<\infty)$ we have the trivial bound
 $$
 \sum_{n=1}^\Delta \left(\frac{D}{n}\right)\frac{1}{n} < 1 + \log \Delta.
 $$
The Kronecker symbol $\left(\frac{D}{n}\right)$ is a nonprincipal character modulo $D$. In particular, for any $\Delta$, we have the elementary uniform bound $|s_\Delta|\leq D$, where
 $$
 s_\Delta = \sum_{n=1}^\Delta \left(\frac{D}{n}\right).
 $$
Then
 $$
\sum_{n=\Delta+1}^\infty  \left(\frac{D}{n}\right)\frac{1}{n}=\sum_{n=\Delta+1}^\infty \frac{s_n-s_{n-1}}{n} =\sum_{n=\Delta+1}^\infty \frac{s_n}{n(n+1)} -\frac{s_\Delta}{\Delta+1},
 $$
which we can bound by
 \begin{align*}
 D \int_{\Delta}^\infty \frac{dt}{t(t+1)} + \frac{D}{\Delta+1}  = D\int_{\Delta}^{\Delta+1} \frac{dt}{t} + \frac{D}{\Delta+1} \ll \frac{D}{\Delta}.
 \end{align*}
 Choosing $\Delta=D$ we conclude that $|L(\chi_D,1)|\ll \log D$.
 \end{proof}

 \begin{remark}
     Stronger estimates for $s_\Delta$ are available (e.g., the theorems of P\'olya--Vinogradov or Burgess) but do not yield a better final result. However, as noted earlier, we can replace $\log N$ by $\log\log N$ if GRH is true.
 \end{remark}
 
The results of this section show that 
$$
\liminf_{N\to\infty} \frac{\log{\rm cr}(\mathcal{C}(N))}{\log N} \leq 5
$$
and together with the second statement of Theorem \ref{crossingnumber}, proved in the previous section, this concludes the proof of Theorem \ref{crossingnumber}.

\section{Subsystems with optimal crossing number}

So far, we have seen that the crossing number of curve systems with about $g^{1+\alpha}$ curves grows at least as $g^{1+2\alpha}$. Moreover, this growth rate is asymptotically achieved by systems of systoles of surfaces associated with congruence lattices. In the following, we explain why these systems of systoles contain subsystems with smaller growth rate and asymptotically minimal crossing number.

\begin{lemma}
\label{average}
Let $\mathcal{C}$ be a family of $n$ curves on a given closed surface. Then for every $k<n$ there is a subfamily of $\mathcal{C}$ of size $k$ with crossing number at most $(k/n)^2\text{cr}(\mathcal{C})$.
\end{lemma}

\begin{proof}
In fact, the average crossing number of subfamilies of size $k$ of $\mathcal{C}$ is precisely $$\frac{k(k-1)}{n(n-1)}\text{cr}(\mathcal{C})< (k/n)^2 \text{cr}(\mathcal{C}).$$ 
To see this, let $F$ be the family of subsets of $\mathcal{C}$ of size $k$ and note that every pair $v,w \in \mathcal{C}$ appears in precisely $n-2 \choose k-2$ members of $F$. This implies that 
$$\sum_{\mathcal{B} \in F} \text{cr}(\mathcal{B}) = {n-2 \choose k-2} \text{cr}(\mathcal{C}).$$
Since $F$ has $n \choose k$ members, the average of $\text{cr}(\mathcal{B})$ over $\mathcal{B} \in F$ is 
$$\frac{\binom{n-2}{k-2}}{\binom{n}{k}} \text{cr}(\mathcal{C}) = \frac{k(k-1)}{n(n-1)}\text{cr}(\mathcal{C}).$$
\end{proof}

Now let $0<\beta<\alpha$, and suppose there exists a family of curves of size $g^{1+\alpha}$ and crossing number at most $g^{1+2\alpha}$ on a given closed surface of genus $g$. Then there is a subfamily of size $g^{1+\beta}$ and at most $g^{1+2\beta}$ intersections. Indeed, by Lemma~\ref{average} we get a subfamily $\mathcal{B}$ of size $g^{1+\beta}$ such that 
$$\text{cr}(\mathcal{B})\leq \Big(\frac{g^{1+\beta}}{g^{1+\alpha}}\Big)^2 \text{cr}(\mathcal{C}) = g^{2\beta - 2 \alpha} \text{cr}(\mathcal{C}) \leq g^{2\beta - 2 \alpha} g^{1+2\alpha} = g^{1+2\beta}.$$

Applying this to the systems of systoles of surfaces associated with congruence lattices, we obtain curve systems of size $g^{1+\beta}$ with minimal crossing number, for all $\beta \leq 1/3$. We intend to deal with the case $\beta \geq 1/3$ in a separate article.


\smallskip
\noindent
\texttt{sebastian.baader@unibe.ch}

\smallskip
\noindent
\texttt{claire.burrin@math.uzh.ch}

\smallskip
\noindent
\texttt{luca.studer@bhf.ch}


\begin{thebibliography}{9}




\bibitem{FM}
     B.~Farb, D.~Margalit: \emph{A primer on mapping class groups}, Princeton Mathematical Series, 49. Princeton University Press, Princeton, NJ, 2012.

\bibitem{JungSardari2022} 
     J.~Jung and N.T.~Sardari: \emph{Intersecting geodesics on the modular surface}, Algebra and Number Theory 17 (2023), no.~7, 1325--1357.


\bibitem{Pa}
     H.~Parlier: \emph{Kissing numbers for surfaces}, J.~Topol.~6 (2013), no.~3, 777--791.



\bibitem{Sch}
     P.~Schmutz Schaller: \emph{Extremal Riemann surfaces with a large number of systoles}, Contemp. Math.~201, Amer. Math. Soc., Providence, RI, 1997. 



\end{thebibliography}
\end{document}